\documentclass[12pt]{amsart}    
\usepackage{amscd}

\def\cal#1{\mathcal{#1}}

\def\NZQ{\Bbb}               

\def\QQ{{\NZQ Q}}
\def\ZZ{{\NZQ Z}}
\def\PP{{\NZQ P}}

\def\CC{{\NZQ C}}

\def\PP{{\NZQ P}}

\def\frk{\frak}               

\def\mm{{\frk m}}

\def\opn#1#2{\def#1{\operatorname{#2}}} 
\opn\chara{char}
\opn\length{\ell}
\opn\pd{pd}
\opn\rk{rk}
\opn\projdim{proj\,dim}
\opn\rank{rank}
\opn\depth{depth}
\opn\grade{grade}
\opn\height{ht}
\opn\embdim{emb\,dim}
\opn\codim{codim}
\def\OO{\mathcal{O}}
\opn\Tr{Tr}
\opn\bigrank{big\,rank}
\opn\superheight{superheight}\opn\lcm{lcm}
\opn\trdeg{tr\,deg}%
\opn\reg{reg}
\opn\lreg{lreg}
\opn\Div{Div}
\opn\WDiv{WDiv}
\opn\cl{cl}
\opn\Cl{Cl}
\opn\Spf{Spf}
\opn\Spec{Spec}
\opn\Supp{Supp}
\opn\supp{supp}
\opn\Sing{Sing}
\opn\Ass{Ass}
\opn\Assh{Assh}
\opn\Min{Min}
\opn\Reg{Reg}
\opn\Ann{Ann}
\opn\Rad{Rad}
\opn\Soc{Soc}
\opn\Socle{Socle}
\opn\Ker{Ker}
\opn\Coker{Coker}
\opn\Im{Im}
\opn\Hom{Hom}
\opn\Tor{Tor}
\opn\Ext{Ext}
\opn\End{End}
\opn\Aut{Aut}
\opn\id{id}

\opn\nat{nat}
\opn\pff{pf}
\opn\Pf{Pf}
\opn\GL{GL}
\opn\SL{SL}
\opn\mod{mod}
\opn\ord{ord}
\opn\Proj{Proj}
\opn\aff{aff}
\opn\con{conv}
\opn\relint{relint}
\opn\st{st}
\opn\lk{lk}
\opn\cn{cn}
\opn\core{core}
\opn\vol{vol}
\opn\link{link}
\opn\star{star}
\opn\gr{gr}

\def\pot#1#2{#1[\kern-0.28ex[#2]\kern-0.28ex]}

\opn\dirlim{\underrightarrow{\lim}}
\opn\inivlim{\underleftarrow{\lim}}
\let\dirsum=\oplus
\let\tensor=\otimes
\let\iso=\cong

\let\mcone= * 
\let\to=\rightarrow

\let\To=\longrightarrow

\def\Implies{\ifmmode\Longrightarrow \else
     \unskip${}\Longrightarrow{}$\ignorespaces\fi}
\def\implies{\ifmmode\Rightarrow \else
     \unskip${}\Rightarrow{}$\ignorespaces\fi}
\def\iff{\ifmmode\Longleftrightarrow \else
     \unskip${}\Longleftrightarrow{}$\ignorespaces\fi}

\let\:=\colon
\opn\H{H}
\opn\Pic{Pic}

\newtheorem{Theorem}{Theorem}
\newtheorem{Lemma}[Theorem]{Lemma}
\newtheorem{Corollary}[Theorem]{Corollary}
\newtheorem{Proposition}[Theorem]{Proposition}

\newtheorem{Remark}[Theorem]{Remark}

\let\epsilon\varepsilon
\def\OO{{\cal O}} 

\opn\inii{in}
\opn\inim{inm}
\opn\set{set}
\opn\sgn{sgn}

\def\pnt{{\raise0.5mm\hbox{\large\bf.}}}

\begin{document}

\title{An example of non-simply connected non-liftable Calabi-Yau 3-fold
in positive characteristic}

\author{Yukihide Takayama}
\address{Yukihide Takayama, Department of Mathematical
Sciences, Ritsumeikan University, 
1-1-1 Nojihigashi, Kusatsu, Shiga 525-8577, Japan}
\email{takayama@se.ritsumei.ac.jp}

\def\Coh#1#2{H_{\mm}^{#1}(#2)}
\def\eCoh#1#2#3{H_{#1}^{#2}(#3)}

\newcommand{\AppTh}{Theorem~\ref{approxtheorem} }
\def\da{\downarrow}
\newcommand{\ua}{\uparrow}
\newcommand{\namedto}[1]{\buildrel\mbox{$#1$}\over\rightarrow}
\newcommand{\bdel}{\bar\partial}
\newcommand{\proj}{{\rm proj.}}

\maketitle

\newenvironment{myremark}[1]{{\bf Note:\ } \dotfill\\ \it{#1}}{\\ \dotfill
{\bf Note end.}}
\newcommand{\transdeg}[2]{{\rm trans. deg}_{#1}(#2)}
\newcommand{\mSpec}[1]{{\rm m\hbox{-}Spec}(#1)}

\newcommand{\tbf}{{{\Large To Be Filled!!}}}

\pagestyle{plain}
\maketitle
\def\gCoh#1#2#3{H_{#1}^{#2}\left(#3\right)}
\def\subsetneq{\raisebox{.6ex}{{\small $\; \underset{\ne}{\subset}\; $}}}


\begin{abstract}
  We show an example of non-simply connected non-liftable Calabi-Yau threefold
  over an algebraically closed field of characteristic $3$.
  It is constructed from a simply connected example by S.~Schr\"oer.
\end{abstract}

\section{Introduction}

Let $k$ be an algebraically closed field of positive characteristic
$p>0$. A smooth projective variety $X$ over $k$ is called Calabi-Yau if it has
trivial canonical bundle $\omega_X\iso \OO_X$ and $H^i(X,\OO_X)=0$ for
$i\ne 0, \dim X$.

A projective variety $X$ over $k$ is called
projectively liftable
to characteristic $0$
 (liftable, for short) if there exists a
projective scheme
\begin{equation*}
  {\frk X}\To \Spec R
  \end{equation*}
over a complete discrete valuation ring $(R,\mm)$ of characteristic
$0$ such that $k=R/\mm$ and $X\iso
{\frk X}\tensor_{R} k$.
It is known that Calabi-Yau varieties of dimension
$\leq 2$ are liftable, but for dimension~3 several non-liftable
examples have been found so far.

Examples of non-liftable 
Calabi-Yau threefolds by Hirokado, Ito, Saito
and Schr\"oer \cite{H99, H01, HISark, HISmanuscripta, Schr04} are all
simply connected and non-liftability is caused by vanishing the third
Betti number $b_3(X)=0$, which never holds in characteristic $0$.
The examples by Cynk and
van~Straten~\cite{CyvS} are double cover of $\PP^3$ ramified at an
arrangement of 8 planes of characteristic $p=3,5$. The third Betti
number is non-trivial at least for $p=5$. The author is not aware
whether the examples by Cynk and van Straten are simply connected
or the third Betti number vanishes for $p=3$.

In this paper, we consider a positive characteristic version of
Beauville's method to construct non-simply connected Calabi-Yau
threefold over $\CC$. This method virtually allow us to produce many non-simply
connected non-liftable Calabi-Yau threefolds with trivial third Betti number,
provided that we find suitable finite group action.

As an instance we give an example of non-liftable Calabi-Yau threefold
produced from the example by Schr\"oer in characteristic $p=3$.

\section{Calabi-Yau \'etale covers}

In this section, we fix an algebraically closed field $k$ of positive
characteristic $p>0$. For a sheaf ${\cal F}$ over a  variety $X$,
we denote $\dim_k H^i(X, {\cal F})$  by $h^i(X, {\cal F})$.
Also $h^{ij}(X)$ denotes the Hodge number $\dim_k H^j(X, \Omega^i_{X/k})$.

Beauville~\cite{BeauCY} constructed an example of non-simply connected
Calabi-Yau threefold over $\CC$ using free action of finite group.
This idea is also used in \cite{BD08}
for constructing many examples of non-simply connected
Calabi-Yau threefolds over $\CC$ from Schoen type Calabi-Yau threefold.
Here we give a slightly extended version of Beauville's method in
positive characteristic.

\begin{Lemma}
  \label{leray}
  Let $d$ be a positive integer
  with $(d, p)=1$,
  $f:X\to Y$ a finite \'{e}tale cover of degree $d$,
  where $X$ and $Y$ are varieties,
  and $F$ a quasi-coherent sheaf on $Y$. Then
  $\dim_kH^i(Y,F)\leq \dim_k H^i(X, f^*F)$.
\end{Lemma}
\begin{proof}
  Since $f$ is surjective, we have a canonical injection
  $f^\sharp:
  \OO_Y \hookrightarrow f_*\OO_X$.
  On the other hand, since $d$ is a unit in the base field, we have
  the trace map $Tr: f_*\OO_X = f_*f^*\OO_Y\To \OO_Y$ whose composition
  with the canonical morphism $\psi: \OO_Y\to f_*f^*\OO_Y$ yields
  multiplication map by $d$: $\OO_Y\to f_*f^*\OO_Y \overset{Tr}{\to}\OO_Y$.
  Thus we have a split exact sequence
  \begin{equation*}
    0\To \OO_Y \To f_*\OO_X \To \Coker\psi\To 0
  \end{equation*}
  hence $f_*\OO_X \iso \OO_Y\dirsum \Coker {\psi}$
  and we have $f_*f^*F \iso (f_*\OO_X)\tensor F = F\dirsum L$,
  with $L:= \Coker\psi\tensor F$ by projection formula.
  Now since $f$ is a finite morphism, we have 
  $H^i(X, f^*F) \iso H^i(Y, f_*f^*F) \iso H^i(Y, F)\dirsum H^i(Y,L)$
  and we obtain the required result.
\end{proof}

\begin{Theorem}[cf. Beauville~\cite{BeauCY}]
  \label{beauvillePC}
  For an odd integer $n\geq 3$ and a Calabi-Yau $n$-fold $X$,
  if there exists a finite \'{e}tale cover $f:X\to Y$
  of degree $d$ prime to $p$ 
  to a projective variety $Y$,
  then $Y$ is also Calabi-Yau.
\end{Theorem}

\begin{proof}
  Since $f$ is finite \'{e}tale of degree $d$, $(d,p)=1$,
  and $h^i(X, \OO_X)=0$ for $i\ne{0, n}$, 
  we have $h^i(Y, \OO_Y)=0$ for $i\ne{0,n}$ by Lemma~\ref{leray}.
  Moreover,
  since $f$ is \'{e}tale and $X$ is smooth projective,
  $Y$ is also smooth projective.
 %
  Now since $X$ is Calabi-Yau and $n$ is odd, we have $\chi(\OO_X)=0$
  and
  $h^0(Y, \OO_Y) - h^n(Y, \OO_Y)=\chi(\OO_Y)
  = \deg(f)\cdot \chi(\OO_X) =0$
  so that $h^n(Y,\OO_Y) = h^0(Y,\OO_Y)=1$
  since $k$
  is algebraically closed. By Serre duality, we have
  $h^0(Y, \omega_Y)=1$, so that we can take a non-trivial global section
  $\sigma\in H^0(\omega_Y)$. Then its pullback
  $f^*\sigma = \sigma\circ f  \in H^0(\omega_X)$ is everywhere nonzero.
  Since $f$ is surjective, $\sigma$ is also
  everywhere nonzero, which means that $\omega_Y \iso \OO_Y$.
\end{proof}

Note that,
Theorem~\ref{beauvillePC} does not hold
for even-dimensional Calabi-Yau variety.
In fact, following the discussion of the proof of
Theorem~\ref{beauvillePC},
we have
  \begin{equation*}
    1 + \dim_kH^0(\omega_Y) = \chi(\OO_Y) = \deg f\cdot \chi(\OO_X) =2 \deg(f)
  \end{equation*}
Hence, we have $\omega_Y\iso \OO_Y$ only if $\deg(f)=1$, i.e., $X\iso Y$.

\begin{Corollary}
  \label{construction}
  Let $n \geq 3$ be an odd integer.
  If $X$ is simply connected Calabi-Yau $n$-fold and
  $Y = X/G$ is projective, where $G$ is a finite group
  with $\sharp{G}=d$ and $(d,p)=1$ acting on $X$ freely.
  Then $Y$ is Calabi-Yau
  with the algebraic fundamental group $\pi_1^{alg}(Y)= G$.
\end{Corollary}

Notice that we assume Calabi-Yau variety to be projective and
existence of \'{e}tale cover $f:X\To Y$ with $X$ projective
does not necessarily imply $Y$ to be projective.
Hence we have to assume projectivity
of the quotient $Y=X/G$ in Cor.~\ref{construction}.

There are Calabi-Yau threefolds $X$ in positive characteristic
with the third Betti number $b_3(X):= \dim_{\QQ_\ell}H^3_{et}(X, \QQ_\ell)=0$,
with $\ell\ne p$, \cite{H99,Schr04,HISark,HISmanuscripta}.
Since the Betti numbers are preserved in deformation and,
by Hodge decomposition, 
a Calabi-Yau threefold in characteristic~0 always has non-trivial third Betti
number, Calabi-Yau threefolds $X$ in positive characteristic
with $b_3(X)=0$ cannot be liftable to characteristic~$0$.
Now we have

\begin{Corollary}
  \label{concrete}
  Let $X$ be a non-liftable Calabi-Yau threefold $X$ over an
  algebraically closed field
  $k$ of $\chara(k)=p>0$ with $b_3(X)=0$
  and $f:X\to Y$ a finite \'{e}tale cover of a projective
  variety $Y$ with $\deg{f} = d$, $(d,p)=1$.
  Then $Y$ is also a non-liftable Calabi-Yau threefold
  with $b_i(Y)\leq b_i(X)$ for all $i$
  and $h^{ij}(Y)\leq h^{ij}(X)$ for all $i,j$.
  If, furthermore, $X$ is simply connected and
  $Y = X/G$ with $G$ a finite group acting freely on $X$
  with $\sharp G$ prime to $p$, then $Y$ is non-simply connected
  with $\pi_1^{alg}(Y)=G$.
\end{Corollary}
\begin{proof}
  $\QQ_{\ell}$ is a constant sheaf so that $f^*\QQ_\ell = \QQ_\ell$.
  Moreover, since $f:X\to Y$ is \'{e}tale, we have
  $f^*\Omega_{Y/k}^i \iso \Omega_{X/k}^i$
  for all $i$.
  Hence by Lemma~\ref{leray} we have
  $h^{ij}(Y)\leq h^{ij}(X)$.
  Also for the \'{e}tale sheaf $\QQ_\ell$, we have a composite map
  \begin{equation*}
    H^i_{et}(Y, \QQ_\ell) \To H^i_{et}(X, \QQ_\ell)
    \overset{Tr}{\To} H^i_{et}(Y,\QQ_\ell)
  \end{equation*}
  where $Tr$ is the trace map (cf. Lemma~V.1.12~\cite{Milne}).
  This map is multiplication by $d$, which is invertible
  since  $(d,p)=1$.
  Thus we have 
  $\dim_{\QQ_\ell}H^i_{et}(Y, \QQ_\ell)\leq \dim_{\QQ_\ell}H^i(X, \QQ_\ell)$.
  In particular we have $H^3_{et}(Y, \QQ_\ell)=0$
   from $H^3_{et}(X, \QQ_{\ell})=0$.
\end{proof}

\section{Non-liftable CY 3-fold with fundamental group $\ZZ/2\ZZ$}

The aim of this section is to show
the following result using the method given in the previous section.

\begin{Theorem}
  \label{main}
  There exists an example of non-liftable Calabi-Yau threefold
  $X$ over an algebraically
closed field of characteristic $3$ such that $b_3(X)=0$ and $\pi_1^{alg}(X)
= \ZZ/2\ZZ$.
\end{Theorem}

We construct our example from
the example by Schr\"oer in characteristic$~3$ \cite{Schr04}.
The idea is to apply Liebermann's involution on schr\"oer variety.

\subsection{Review of Schr\"oer variety}

We first recall the construction of Schr\"oer variety in characteristic $p=3$.
It is obtained from a Moret-Bailly's fibration of supersingular
abelian surfaces over $\PP^1$ \cite{MB} by taking the desingularized
quotient of involution. Thus it is a fibration of Kummer surfaces
over $\PP^1$.

Precisely, let $k$ be an algebraically closed field of characteristic
$p=3$.
Take a superspecial abelian surface $A = E_1\times E_2$,
where $E_i$ are supersingular elliptic curves. As an abelian variety,
we have a subgroup scheme
$K(L):= \alpha_p\times\alpha_p\subset A$, where $\alpha_p$ is
the group scheme defined by 
$\alpha_p = \Spec k[X]/(X^p)$, and we take the product with
$\PP^1$: $K(L)\times {\PP^1} \subset A\times \PP^1$.
Now we have a subgroup scheme $H:=\alpha_p\times \PP^1 \subset K(L)\times\PP^1$
which is a group scheme of height one  corresponding
to $p$-Lie subalgebras
$\OO_{\PP^1}(-1)\subset \OO_{\PP^1}\dirsum\OO_{\PP^1} (\subset Lie(A\times\PP^1)$.
Explicitly, we have $K(L)\times \PP^1
=\Spec \OO_{\PP^1}[x,y]/(x^p,y^p)$
and $H \iso \Spec \OO_{\PP^1}[x,y]/(x^p,y^p, tx-sy)$
where $[s:t]$ is the projective coordinates of $\PP^1$.
Now take $\tilde{X}:= (A\times\PP^1)/H$ which is supersingular abelian
surface over $\PP^1$. Now by taking the fiberwise quotient by the involution
induced from $(x,y)\mapsto (-x,-y)$ on $A$ together with desingularization
we obtain the Kummer surface pencil $\pi:X\to \PP^1$, which is a 
Calabi-Yau threefolds with $b_3(X)=0$, i.e., non-liftable.

\subsection{Liebermann's involution and our example}

We recall Liebermann's involution (\cite{MN84, P80}).
Using the notation in the previous subsection,
let $a=(\alpha, \beta)\in A$ be a 2-torsion point
not lying on $E_1$ or $E_2$, i.e.,
$\alpha$ and $\beta$ are non-zero $2$-torsion points in
each $E_i$, $(i=1,2)$.
Let $\tilde{S}$ be a Kummer surface obtained from $A$
and $\sigma_R$ an involution of $\tilde{S}$
induced from the involution $(x,y)\mapsto (-x,y)$ on $A$.
Also let $\sigma_K$ be also an involution of $\tilde{S}$
induced from the translation by $a$:
$A\ni x \mapsto x + a \in A$.
Then $\epsilon := \sigma_R\sigma_K$ is a fixed point free
involution on $\tilde{S}$. It is explicitly described
at the level of $A$ by
\begin{equation*}
  (x,y)\mapsto (-x + \alpha, y + \beta).
\end{equation*}
Then the  quotient $\tilde{S}/\epsilon$
is an Enriques surface.

\subsection{Proof of Theorem~\ref{main}}

Now our example is obtained by applying $\epsilon$ not on
$A$ (and its lifting $A\times\PP^1$)
but $A/\alpha_p$ (and its lifting $A\times \PP^1/H$).
Notice that, since $\epsilon$ is involution and the action of $H$
on each fiber is
the translation by an order $p=3$ element, $\epsilon$ acts freely on
$A\times \PP^1/H$ and on the Schr\"oer variety obtained from this.
Then we have
\begin{equation*}
  \begin{CD}
    X       @>{f}>>   Y:= X/\epsilon\\
    @V{\pi}VV          @V{\varphi}VV  \\
    \PP^1   @=    \PP^1
    \end{CD}
  \end{equation*}
where $f$ is an \'{e}tale cover of degree~2 and
$\varphi$ is defined by $\varphi\circ f = \pi$.

Thus we have the following, from which Theorem~\ref{main} is immediate.

\begin{Proposition}
  For the Schr\"oer variety $X$ in characteristic~3,
  there exists an \'{e}tale cover $f:X\to Y$ of degree 2.
  $Y$ is a non-liftable Calabi-Yau threefold
  with $\pi_1^{alg}(Y) = \ZZ/2\ZZ$ and $b_1(Y)=b_3(Y)=b_5(Y)=0$ and
  $h^{ij}(Y) = h^{ji}(Y)=0$ for $(i,j)=(0,1), (0,2), (1,2), (1,3), (2,3)$.
\end{Proposition}
\begin{proof}
According to \cite{Schr04, Ek}, we have $b_1(X)=b_3(X)=b_5(X)=0$
and $h^{ij}(X) = h^{ji}(X)=0$ for $(i,j)=(0,1), (0,2), (1,2),
(1,3), (2,3)$ so that by Corollary~\ref{concrete} the same holds for $Y$.
For the projectivity of $Y$
we have only to prove that
  $\varphi$ is a projective morphism.
  Since $X_\eta := \pi^{-1}(\eta)$ with $\eta\in\PP^1$
  a generic point, is a smooth K3 over $k(\eta)$ and
  $Y_\eta := \varphi^{-1}(\eta)$ is the Enriques surface
  over $k(\eta)$ obtained by taking the quotient of $X_\eta$ by
  the involution. Thus $Y_\eta$ is a smooth projective variety.
  Let ${\cal L}_\eta$ be an ample line bundle of $Y_\eta$.
  By taking the Zariski closure, we can extend ${\cal L}_\eta$
  to a line bundle ${\cal L}$ on $Y$. Since the Picard numbers $\rho(Y_\eta)
  = \rho(Y_t)$
  ($t\in\PP^1$), the specialization map
  $\Pic(Y_\eta)\To \Pic(Y_t)$ has finite cokernel. Thus
  ${\cal L}.C > 0$ for any  irreducible curve $C\subset Y_t$.
  We also have ${\cal L}_t.{\cal L}_t = {\cal L}_\eta.{\cal L}_\eta>0$
    and thus ${\cal L}_t = {\cal L}\vert_{Y_t}$ is ample.
    Therefore, ${\cal L}$ is $\varphi$-ample.
  \end{proof}

\begin{Remark}
Other examples of non-liftable Calabi-Yau threefolds by Hirokado, Ito
and Saito \cite{H01, HISark,HISmanuscripta} are all desingularized
fiber product of two quasi-elliptic surfaces or fiber product of
elliptic and quasi-elliptic rational surfaces.  A quasi-elliptic
curve is not abelian variety and moreover for a quasi-elliptic curve
$C$ over an algebraically closed field $k$ of characteristic $3$,
$\Aut(C)$ does not contain an element of order~$2$
(Proposition~6~\cite{BMIII}).  Thus we cannot use the Liebermann's
involution as we did for the Schr\"oer variety.
\end{Remark}


\begin{thebibliography}{99}

\bibitem{BeauCY} Beauville, A.
  A Calabi-Yau threefold with non-abelian
  fundamental group. New trends in algebraic geometry
  (Warwick, 1996), 13-17, London Math. Soc. Lecture Note Ser., 264,
  Cambridge Univ. Press, Cambridge, 1999.

\bibitem{BMIII}Bombieri,~E. Mumford,~D.
  Enriques' classification of surfaces in char. $p$. III.
  Invent. Math.  35  (1976), 197--232.
  
\bibitem{BD08} Bouchard, V. and Donagi,R.
  On a class of non-simply connected Calabi-Yau 3-folds,
  Commun. Number Theory Phys. 2 (2008) no. 1, 1--61.

\bibitem{CyvS} Cynk, S. and van Straten, D. Small resolutions
  and non-liftable Calabi-Yau threefolds. Manuscripta Math.  130
  (2009), no. 2, 233--249.

\bibitem{Ek} Ekedahl, T. On non-liftable Calabi-Yau threefolds (preprint),
  math.AG/0306435
  
\bibitem{H99} Hirokado, M. A non-liftable Calabi-Yau threefold
  in characteristic $3$. Tohoku Math. J. (2) 51 (1999), no. 4,
  479--487.

\bibitem{H01} Hirokado, M. Calabi-Yau threefolds obtained as fiber
  products of elliptic and quasi-elliptic rational surfaces.
  J.~Pure Appl. Algebra 162 (2001) no.2-3, 251-271.
  
  
\bibitem{HISark}Hirokado, Masayuki; Ito, Hiroyuki; Saito,
  Natsuo. Calabi-Yau threefolds arising from fiber products of
  rational quasi-elliptic surfaces. I. Ark. Mat. 45 (2007), no. 2,
  279--296.
  
\bibitem{HISmanuscripta} Hirokado, Masayuki; Ito, Hiroyuki; Saito,
  Natsuo.  Calabi-Yau threefolds arising from fiber products of
  rational quasi-elliptic surfaces. II. Manuscripta Math. 125 (2008),
  no. 3, 325--343.

\bibitem{Milne} Milne, J. {\em Etale cohomology}. Princeton
  Mathematical Series, 33. Princeton University Press, Princeton,
  N.J., 1980.

\bibitem{MB} Moret-Bailly, L. Families de courbes et de vari\'{e}t\'{e}s
  ab\'{e}liennes sur $\PP^1$, I.~II.
  Ast\'{e}risque~86, 1981. pp.109--140.

\bibitem{MN84} Mukai, S. and Namikawa, Y.
  Automorphism of Enriques surfaces which act trivially on the
  cohomology groups, Invent.~Math.~77, 383--397, 1984.

\bibitem{P80} Peters, C. On automorphism of compact K\"ahler surfaces.
  Journ\'{e}es de G\'{e}ometrie Alg\'{e}brique d'Angers, Juillet
  1979/Algebraic Geometry, Angers, 1979, pp. 249--267, Sijthoff \&
  Noordhoff, Alphen aan den Rijn--Germantown, Md., 1980.

\bibitem{Scho09} Schoen, C. Desingularized fiber products of semi-stable
  elliptic surfaces with vanishing third Betti number.
  Compositio Math.~145 (2009) 89-111

\bibitem{Schr04} Schr\"{o}er, S. Some Calabi-Yau threefolds with
  obstructed deformations over the Witt vectors. Compos. Math.  140
  (2004), no. 6, 1579--1592.

\bibitem{Tca} Takayama, Y. Kodaira type vanishing theorem for the
  Hirokado variety. Commun.~Algebra~42 (2014), 4744--4750.
\end{thebibliography}
\end{document}